\documentclass[12pt]{article}
\usepackage[pdftex]{color, graphicx}
\usepackage{setspace}
\usepackage[utf8]{inputenc}
\usepackage{amsmath}
\usepackage[pdftex,bookmarks,colorlinks]{hyperref}
\usepackage{latexsym}
\usepackage{amsfonts}
\usepackage{amssymb}
\usepackage{amsthm}
\usepackage{paralist}
\usepackage{mathrsfs}
\begin{document}
\newtheorem{thm}{Theorem}
\newtheorem{cor}[thm]{Corollary}
\newtheorem{lem}{Lemma}
\theoremstyle{remark}\newtheorem{rem}{Remark}
\theoremstyle{definition}\newtheorem{defn}{Definition}

\title{Variaiton and $\lambda$-jump inequalities on $H^p$ spaces}
\author{Sakin Demir}
%\date{March 10, 2016}
\author{Sakin Demir\\
Agri Ibrahim Cecen University\\ 
Faculty of Education\\
Department of Basic Education\\
04100 A\u{g}r{\i}, Turkey\\
E-mail: sakin.demir@gmail.com
}

%\date{}

\maketitle

%% Classification and key words; note that the 2010 classification is used:

\renewcommand{\thefootnote}{}



\renewcommand{\thefootnote}{\arabic{footnote}}
\setcounter{footnote}{0}
\maketitle
\begin{abstract}Let $\phi\in \mathscr{S}$ with $\int\phi (x)\, dx=1$, and define
$$\phi_t(x)=\frac{1}{t^n}\phi (\frac{x}{t}),$$
and denote the function family $\{\phi_t\ast f(x)\}_{t>0}$ by $\Phi\ast f(x)$. Suppose that there exists a constant $C_1$ such that
$$\sum_{t>0} |\hat{\phi}_t(x)|^2<C_1$$
for all $x\in \mathbb{R}^n$. Then
\begin{enumerate} [(i)]
\item There exists a constant $C_2>0$ such that
$$\|\mathscr{V}_2(\Phi\ast f)\|_{L^p}\leq C_2\|f\|_{H^p},\;\;\frac{n}{n+1}<p\leq 1$$
for all $f\in H^p(\mathbb{R}^n)$, $\frac{n}{n+1}<p\leq 1$.
\item The $\lambda$-jump operator $N_{\lambda}(\Phi\ast f)$ satisfies
$$\|\lambda [N_{\lambda}(\Phi\ast f)]^{1/2}\|_{L^p}\leq C_3\|f\|_{H^p},\;\;\frac{n}{n+1}<p\leq 1,$$
uniformly in $\lambda >0$ for some constant $C_3>0$.
\end{enumerate}
\end{abstract}

%%%%%%%%%%%%%%%%%%%%%

Variation, oscillation and $\lambda$-jump inequalities on $L^p$ spaces have long been the research subjects of many mathematicians in probability, harmonic analysis and ergodic theory (see \cite{lpgle}, \cite{sdemir}, \cite{jbrgn}, \cite{rjkm}, and \cite{jsw}).\\
 When $0<p\leq 1$ the Hardy space $H^p$ does no longer behave like $L^p$, that's why proving variation and oscilation inequalities on $H^p$ spaces for $0<p\leq 1$ requires a completely different work than what one does when working on an $L^p$ space.\\ 
It has been proved in the  author's Ph.D thesis (see \cite{sdem}) that several operators including the ergodic square function of differences of averages over lacunary sequences map  ergodic $H^1$  to $L^1$.\\
Let $\phi\in \mathscr{S}$ with $\int\phi (x)\, dx=1$, and define
$$\phi_t(x)=\frac{1}{t^n}\phi (\frac{x}{t}),$$
and denote the function family $\{\phi_t\ast f(x)\}_{t>0}$ by $\Phi\ast f(x)$.
Let $0<p<\infty$. A distribution $f$ belongs to $H^p(\mathbb{R}^n)$ if the maximal function
$$Mf(x)=\sup_{t>0}|(\phi_t\ast f(x)|$$
is in $L^p$.\\
\begin{defn} Let $\mathcal{J}$ be a subset of $\mathbb{R}$ (or more generally an ordered index set). We consider real or complex valued functions $t\mapsto a_t$ defined on $\mathcal{J}$ and define their $\rho$-variation as
$$\|a\|_{v_{\rho}}=\sup\left(\sum_{i\geq 1} |a_{t_{i+1}}-a_{t_i}|^{\rho}\right)^{1/\rho},$$
where the supremum runs over all finite decreasing sequences $(t_i)$ in $\mathcal{J}$.
\end{defn}
We can define 
\begin{align*}
\mathscr{V}_{\rho}(\Phi\ast f)(x)&=\|\{\phi_t\ast f(x)\}_{t>0}\|_{v_\rho}\\
&=\sup\left(\sum_{i\geq 1} |\phi_{t_{i+1}}\ast f(x)-\phi_{t_i}\ast f(x)|^{\rho}\right)^{1/\rho}
\end{align*}
for any $x\in\mathbb{R}^n$.\\
Let $F=\{F_t:t\in \mathcal{J}\}$ be a family of Lebesgue measurable functions defined on $\mathbb{R}^n$.\\

The $\lambda$-jump function $N_{\lambda}(F)$ is defined  as the supremum of all integers $N$ for which there is an increasing sequence $0<s_1<t_1\leq s_2<t_2<\cdots \leq s_N<t_N$ such that
$$|F_{t_k}(x)-F_{s_k}(x)|>\lambda $$
for each $k=1,2,3,\dots ,N$.\\
H. Liu~\cite{hliu} studied the variation and $\lambda$-jump inequalities for the function family $\Phi\ast f $ when $\rho >2$, and obtained the following results:

\begin{thm} For any $\rho >2$, there exists a constant $C_\rho>0$ such that
$$\|\mathscr{V}_\rho(\Phi\ast f)\|_{L^p}\leq C_\rho\|f\|_{H^p},\;\;\frac{n}{n+1}<p\leq 1$$
for all $f\in H^p(\mathbb{R}^n)$.
\end{thm}

\begin{thm}\label{sqrfthm} If $\rho >2$, then the $\lambda$-jump operator $N_{\lambda}(\Phi\ast f)$ satisfies
$$\|\lambda [N_{\lambda}(\Phi\ast f)]^{1/\rho}\|_{L^p}\leq C_\rho\|f\|_{H^p},\;\;\frac{n}{n+1}<p\leq 1,$$
uniformly in $\lambda >0$, for all $f\in H^p(\mathbb{R}^n)$.
\end{thm}
In this article we prove the above results for the case $\rho=2$ with an additional assumption on the Fourier transform of $\phi_t$.\\
%%%%%%%%%%%%%%%%%%%%%%%%%%%
We first need to present some known facts related to the atomic decomposition of $H^p(\mathbb{R}^n)$ that will be used when proving our results.
%%%%%%%%%%%%%%%%%%%%%%%%%%%%%%%
\begin{defn} Let $0<p\leq 1\leq q\leq \infty$, $p\neq q$. We say that a function $a\in L^q(\mathbb{R}^n)$ is a $(p,q)$-atom with the center at $c$, if the following conditions are satisfied:
\begin{enumerate} [(i)]
\item $\textrm{supp}\; a\subset B(c, r);$
\item $\|a\|_q\leq |B(c, r)|^{1/q-1/p};$
\item $\int a(x)\, dx=0.$
\end{enumerate}
\end{defn}
The following lemma and its proof can be found in R. H. Latter~\cite{rhl}.

\begin{lem} A distribution $f$ is in $H^p(\mathbb{R}^n)$, $0<p\leq 1$, if and only if there exists a sequence of $(p,q)$-atoms with $1\leq q\leq \infty$ and $q\neq p$, $\{a_j\}$, and a sequence of scalars $\{\lambda_j\}$ such that
$$f=\sum_{j=0}^\infty \lambda_ja_j$$
in the sense of distributions and 
$$C_1\|f\|_{H^p}^p\leq \sum_{j=0}^\infty \lambda_j^p\leq C_2\|f\|_{H^p}^p$$
where $C_1$ and $C_2$ are constants which depend only on $n$ and $p$.
\end{lem}

We will also need the following Lemma when proving our results.
\begin{lem}\label{varsqrbnd} Note that for $\rho \geq 2$ we have the following inequalities
$$\|a\|_{v_\rho}\leq \|a\|_{v_2}\leq 2 \cdot \mathcal{S}(a)$$
where
$$\mathcal{S}(a)=\left(\sum_{t\in\mathcal{J}}|a_t|^2\right)^{1/2}.$$
\end{lem}
\begin{proof} First note that
$$\|a\|_{v_\rho}\leq \|a\|_{v_2}$$
since $\rho \geq 2$. Since in the definition of $\|a\|_{v_2}$ the supremum runs over all finite decrasing sequence $(t_j)\in \mathcal{J}$, and $|x-y|^2\leq 2|x|^2+2|y|^2$, $2\mathcal{S}(a)$ is an upper bound for the sequence 
$$\left\{\left(\sum_{i\geq 1} |a_{t_{i+1}}-a_{t_i}|^2\right)^{1/2}\right\},$$
thus the result follows.
\end{proof}

%%%%%%%%%%%%%%%%%%%%%%%
We can now state and prove our main result.
\begin{thm} Suppose that there exists a constant $C_1$ such that
$$\sum_{t>0} |\hat{\phi}_t(x)|^2<C_1$$
for all $x\in \mathbb{R}^n$. Then there exists a constant $C_2>0$ such that
$$\|\mathscr{V}_2(\Phi\ast f)\|_{L^p}\leq C_2\|f\|_{H^p},\;\;\frac{n}{n+1}<p\leq 1$$
for all $f\in H^p(\mathbb{R}^n)$.
\end{thm}
\begin{proof} First we have
\begin{align*} 
\|\mathscr{V}_2(\Phi\ast f)\|_{L^2}^2&\leq 4\int \sum_{t>0}|\phi_t\ast f|^2\, dx \;\; \;\;\;\textrm{(by Lemma~\ref{varsqrbnd})}\\
&= 4\int \sum_{t>0} |\widehat{\phi_t\ast f}(x)|^2\, dx \;\; \;\;\;\textrm{(by Plancherel's theorem)}\\
&= 4\int \sum_{t>0} |\widehat{\phi_t}(x)\hat{f}(x)|^2\, dx\\
&=4\int \sum_{t>0}|\widehat{\phi_t}(x)|^2|\hat{f}(x)|^2\, dx\\
&\leq 4 C_1\int |\hat{f}(x)|^2\, dx  \;\; \;\;\;\textrm{(by the hypothesis)}\\
&= 4 C_1\int |f(x)|^2\, dx \;\; \;\;\;\textrm{(by Plancherel's theorem)}\\
&=4C_1 \|f\|_{L^2}^2.
\end{align*}
Let now  $a$ be a $(p,2)$ atom and $\frac{n}{n+1}<p\leq 1$. Suppose that $a$ is supported in a cube $Q$, $c$ is the center of $Q$.  By H\"older's inequality and our previous observation we have
\begin{align*}
\int_{4Q}\mathscr{V}_2(\Phi\ast a)^p(x)\, dx&\leq |4Q|^{1-p/2}\|\mathscr{V}_2(f)\|_{L^2}^p\\
&\leq 2\sqrt{C_1} |4Q|^{1-p/2}\|a\|_{L^2}^p\\
&\leq 2\sqrt{C_1}.
\end{align*}
By the Minkowski inequality and the mean value theorem, we see that for $y, \xi \in Q$ and $x\in \mathbb{R}^n\backslash 4Q$,
\begin{align*}
\|\{\phi_t(x-y)-\phi_t(x-\xi )\}_{t>0}\|_{\mathscr{V}_2}&\leq \int_0^\infty\left|\frac{\partial}{\partial t}\left(\phi_t(x-y)-\phi_t(x-\xi)\right)\right|\, dt\\
&\leq C_2|y-\xi |\int_0^\infty\frac{1}{t^{n+2}}\left(1+\frac{|x-\xi |}{t}\right)^{-n-2}\, dt\\
&\leq C_2\frac{|y-\xi |}{|x-\xi |^{n+1}}\int_0^\infty\frac{t^n}{(1+t)^{n+2}}\, dt\\
&\leq C_2\frac{|y-\xi |}{|x-\xi |^{n+1}}.
\end{align*}
Since we also have
$$\mathscr{V}_2(\Phi\ast a)(x)\leq \int_Q|a(y)|\|\{\phi_t(x-y)-\phi_t(x-c )\}_{t>0}\|_{\mathscr{V}_2}\, dy,$$
we obtain
\begin{align*}
\int_{\mathbb{R}^n\backslash 4Q}\mathscr{V}_2(\Phi\ast a)^p(x)\, dx&\leq C_2\int_{|x-c|\geq 4 |Q|}\frac{|Q|^p}{|x-c|^{(n+1)p}}\, dx\left(\int_Q|a(y)|\, dy\right)^p\\
&\leq C_2,
\end{align*}
and this completes our proof.
\end{proof}
\begin{cor}\label{hliuconj} Suppose that there exists a constant $C_1$ such that
$$\sum_{t>0} |\hat{\phi}_t(x)|^2<C_1$$
for all $x\in \mathbb{R}^n$.  Then the $\lambda$-jump operator $N_{\lambda}(\Phi\ast f)$ satisfies
$$\|\lambda [N_{\lambda}(\Phi\ast f)]^{1/2}\|_{L^p}\leq C_2\|f\|_{H^p},\;\;\frac{n}{n+1}<p\leq 1,$$
uniformly in $\lambda >0$, for all $f\in H^p(\mathbb{R}^n)$.
\end{cor}
\begin{proof} It has been proven in \cite{jsw} and easy to see that
$$\lambda [N_{\lambda}(\Phi\ast f)]^{1/\rho}\leq \mathscr{V}_\rho(\Phi\ast f)$$
holds for any $\rho$ uniformly in $\lambda$. Thus the proof follows from Theorem~\ref{sqrfthm}.
\end{proof}
\begin{rem}Note that Corollary~\ref{hliuconj} answers  a conjecture of H.~Liu ~\cite{hliu} affirmatively.
\end{rem}
 %%%%%%%%%%%%%%%%%%%%%%%%%%%%%%%%%%%%%%%%%%%%%%%%

\end{document}